\documentclass{amsart}[11pt]
\usepackage{amsmath,calc}
\usepackage[margin=1.4in]{geometry}
\usepackage{setspace}

\newtheorem{theorem}{Theorem}
\newtheorem{algorithm}{Algorithm}
\newtheorem{lemma}{Lemma}

\newtheorem*{symmetry}{Symmetry}

\theoremstyle{definition}
\newtheorem*{definition}{Definition}

\theoremstyle{remark}

\DeclareMathOperator{\EA}{\mathrm{EA}}

\begin{document}

\title{New algorithms for modular inversion and representation by binary quadratic forms arising from structure in the Euclidean algorithm}
\author{Christina Doran, Shen Lu, Barry R. Smith}
\thanks{This work was supported by a grant from The Edward H. Arnold and Jeanne Donlevy Arnold Program for Experiential Education, which supports research experiences for undergraduates at Lebanon Valley College}

\begin{abstract}
We observe structure in the sequences of quotients and remainders of the Euclidean algorithm with two families of inputs.  Analyzing the remainders, we obtain new algorithms for computing modular inverses and representating prime numbers by the binary quadratic form $x^2 + 3xy + y^2$.  The Euclidean algorithm is commenced with inputs from one of the families, and the first remainder less than a predetermined size produces the modular inverse or representation. 
\end{abstract}
\maketitle

\section{The algorithms}

Intuitively, the iterative nature of the Euclidean algorithm makes the sequences of quotients and remainders ``sensitive to initial conditions''.  A small perturbation to the inputs can induce a chain reaction of increasingly large perturbations in the sequence of quotients and remainders, leading to considerable alterations to both the lengths of the sequences and their entries.   Later entries are especially prone to change because of cumulative effects. 

Our first result, Theorem \ref{T:twoEAs}, provides a surprising example of regularity under perturbation.  When $v$ is a solution of the congruence $v^2 + v - 1 \equiv 0 \pmod{u}$, we show that the Euclidean algorithm with $u$ and $v-1$ always takes one step fewer than the Euclidean algorithm with $u$ and $v$.  The sequences of quotients in both cases are almost identical, differing only in their middle one or two entries.  (They are also symmetric outside of those middle entries.)  We also obtain explicit formulas for the remainders of the Euclidean algorithm with $u$ and $v-1$ in terms of the remainders produced by $u$ and $v$. 

From these formulas we obtain a new algorithm for representing prime numbers by the indefinite quadratic form $x^2 + 3xy + y^2$.  When such a representation exists, the algorithm produces one with $x > y > 0$.  Lemma \ref{L:representation} at the end of this section shows this representation is unique.

\begin{algorithm}\label{A:BQFalgorithm}
Let $p$ be a prime number congruent to $1$ or $4$ modulo $5$.  To compute the unique representation $p = b^2 + 3bc + c^2$ with $b > c > 0$, first compute a solution $v$ to the congruence $ v^2 + v - 1 \equiv 0 \pmod{p}$, then perform the Euclidean algorithm with $p$ and $v$. The first remainder less than $\sqrt{p/5}$ is $c$, and the remainder just preceding is either $b$ or $b+c$.
\end{algorithm}

 This algorithm is similar to earlier algorithms that use the Euclidean algorithm to produce representations by binary quadratic forms \cite{jB1972,gC1908,HMW1990,HMW1990b,kM2002,pW1980}.  Matthew's \cite{kM2002} is the only of these to produce representations by forms with positive discriminant, namely, the forms  $x^2 - wy^2$ with $w = 2$, $3$, $5$, $6$, or $7$.  The algorithm we present is a new contribution to this body of work.

We study a second family of inputs to the Euclidean algorithm, pairs $u > v$ for which $(v \pm 1)^2 \equiv 0 \pmod{u}$. This condition implies that there must exist $a$, $b$, and $c$ with $u = ab^2$ and $v = abc \pm 1$.  Theorem \ref{T:theorem1}  and Theorem \ref{T:theorem4} give an explicit description of the quotients and remainders of the Euclidean algorithm with $u$ and $v$ in terms of the quotients and remainders of the Euclidean algorithm with $b$ and $c$.  

The relationship between the quotients of the Euclidean algorithm with $b$ and $c$ and with $ab^2$ and $abc \pm 1$ is essentially the ``folding lemma'' for continued fractions, first explicated independently in \cite{mF1973} and \cite{jS1979}.  This lemma has inspired a significant body of work concerning the quotients of continued fractions.   These works give attention only to continued fractions -- the remainders in the Euclidean algorithm are never explicitly considered. The description of the entire Euclidean algorithm with $ab^2$ and $abc \pm 1$ in Theorems \ref{T:theorem1} and \ref{T:theorem4} is new. They are unified by Theorem \ref{T:degenerate}, which arithmetically characterizes the quotient pattern that will appear in the Euclidean algorithm with $u$ and $v$ when $(v \pm 1)^2 \equiv 0 \pmod{u}$.

Analysis of the remainders leads to another new algorithm, this time for modular inversion. 

\begin{algorithm}\label{A:inverse}
If $m$ and $n$ are relatively prime positive integers, then the multiplicative inverse of $m$ modulo $n$ is the first remainder less than $n$ when the Euclidean algorithm is performed with $n^2$ and $mn+1$.  
\end{algorithm}

A similar algorithm was obtained by Seysen \cite{mS2005}.  In his algorithm, an integer $f$ is arbitrarily chosen with $f > 2n$, and the Euclidean algorithm is run with $fn$ and $fm+1$.  The algorithm is stopped at the first remainder $r$ less than $f+n$, and the modular inverse of $m$ modulo $n$ is then $r-f$ (which can be negative).  If $f$ were allowed to equal $n$, then this would be similar indeed to the algorithm above.  However, Seysen's algorithm does not work generally in this case.  For instance, with $n=12$ and $m=5$, Seysen's algorithm with $f=12$ would say to run the Euclidean algorithm with $144$ and $61$, stopping at the first remainder less than $24$.  This remainder is $22$, and Seysen's algorithm would output $10$, which is not an inverse for $5$ modulo $12$.  Our algorithm above instead produces the inverse $5$. 

The inputs to Algorithm \ref{A:inverse} are less than half the size of the inputs to Seysen's.  But Seysen's algorithm has the flexibility arising from choosing the factor $f$.  It would be interesting to see if both algorithms can fit in a common framework.

Our results are a new contribution to the literature on algorithmic number theory, but we believe the modular inversion algorithm also has pedagogical value.  Students are less prone to mistakes working by hand with the new algorithm rather than the extended Euclidean algorithm or Blankinship's matrix algorithm \cite{wB1963}.   The new algorithm might seem non-intuitive, but our proof is elementary and is an amalgam of topics encountered by a student learning formal reasoning: the Euclidean algorithm, congruences, and mathematical induction.  

We conclude this section with the result guaranteeing the uniqueness of the representation produced by Algorithm \ref{A:BQFalgorithm}.

\begin{lemma}\label{L:representation}
If $p$ is a prime number congruent to $1$ or $4$ modulo $5$, then there is a unique pair of positive integers $b > c$ satisfying
\begin{equation*}
	p = b^2 + 3bc + c^2.
\end{equation*}
\end{lemma}

\begin{proof}

We work in the field $\mathbb{Q} (\sqrt{5})$.  The algebraic integers in this field are
\begin{equation*}
	\mathcal{O} =  \{ \, \frac{m}{2} + \frac{n}{2}\sqrt{5} \colon m, n, \in \mathbb{Z}, \, m \equiv n \bmod 2 \, \}.
\end{equation*}
 Denote by $\tau$ the nontrivial automorphism of $\mathbb{Q}(\sqrt{5})$ and by $N$ the norm map $N \gamma = \gamma \gamma^{\tau}$. The unit $\varepsilon = \frac{3}{2} +  \frac{1}{2}\sqrt{5}$ generates the group of units of norm 1 in $\mathbb{Z} \left[ \sqrt{5} \right]$. The map
\begin{equation*}
	(b,c) \mapsto \left( b + \tfrac{3}{2} c \right) + \left( \tfrac{1}{2} c \right) \sqrt{5}
\end{equation*}
gives a bijection between all pairs of integers $(b,c)$ with $b^2 + 3bc + c^2 = p$ and all elements of $\mathcal{O}$ of norm $p$. The condition $b > c > 0$  for a pair with $b^2 + 3bc + c^2 = p$ is equivalent to the corresponding element $\tfrac{x}{2} + \tfrac{y}{2} \sqrt{5}$ of $\mathcal{O}$ satisfying  $x > 5y > 0$.

By quadratic reciprocity, $p$ splits in $\mathbb{Q}(\sqrt{5})$. The ring $\mathcal{O}$ is a principal ideal domain, so we may pick a generator $\gamma$ of one of the prime ideals dividing $p$.  Multiplying $\gamma$ by $\tfrac{1}{2} + \tfrac{1}{2} \sqrt{5}$ if necessary, we may assume $\gamma$  has norm $p$.  

There is therefore at least one algebraic integer with norm $p$ of the form $\tfrac{x}{2} + \tfrac{y}{2} \sqrt{5}$.   Among all such elements, let $\alpha$ be one for which $x$ is positive and is as small as possible (i.e., $\alpha$ has minimal positive trace).  Replacing $\alpha$ by $\alpha^{\tau}$ if necessary, we may assume also that $y$ is positive. The lemma will be proved by showing that $\alpha$  is the unique element $\frac{x}{2} + \frac{y}{2} \sqrt{5}$ in $\mathcal{O}$ with norm $p$ and $x > 5y > 0$.  

Define $a_n$, $b_n$ as the integers for which
\begin{align*}
	\alpha \varepsilon^{n} &= \frac{a_n}{2} + \frac{b_n}{2} \sqrt{5}\\
\end{align*}
Then 
\begin{equation*}
	\left( \alpha \varepsilon^{-1} \right)^{\tau} = \frac{3a_0-5b_0}{4} + \frac{a_0 - 3b_0}{4} \sqrt{5}.
\end{equation*}

If we suppose $a_0 - 3b_0 < 0$, then $\tfrac{5b_0 - 3a_0}{4}  > -\frac{1}{3} a_0$.   If $5b_0 - 3a_0$ were negative, then $\left( \alpha \varepsilon^{-1} \right)^{\tau}$ would have norm $p$ and smaller positive trace than $\alpha$, a contradiction.  Thus, again by our choice of $\alpha$, we have $\tfrac{5b_0 - 3a_0}{2} \geq a_0$, hence $a_0 \leq b_0$.  But then 
\begin{equation*}
	N \alpha = \frac{1}{4} \left( a_0^2 - 5 b_0^2 \right) \leq -b_0^2 < 0,
\end{equation*}
which contradicts the assumption that $\alpha$ has norm $p$.

It must be then that $a_0 - 3b_0 > 0$, and thus, $3a_0 - 5b_0 > 0$.  Again using our assumption on $\alpha$, we have $\tfrac{3a_0 - 5b_0}{2} \geq a_0$.  It follows that $a_0 \geq 5b_0 > 0$ (and, in fact, $a_0 > 5b_0$ since $p \neq 5$).

It remains to show that $\alpha$ is the unique algebraic integer $\frac{x}{2} + \frac{y}{2}\sqrt{5}$ with norm $p$ satisfying $x > 5y > 0$. Suppose $x$ and $y$ are integers and set $\frac{w}{2} + \frac{z}{2}\sqrt{5} = (\frac{x}{2} + \frac{y}{2}\sqrt{5}) \varepsilon$.  It is readily checked that if $x > 0$ and $y > 0$, then $w > 0$ and $z > 0$ and $w < 5z$.  It follows that all for all $n \geq 0$, we have $a_n > 0$ and $b_n > 0$, but $a_n > 5 b_n$ only when $n=0$.  Recall that  $\alpha \varepsilon^{-1}  = \frac{a_{-1}}{2} + \frac{b_{-1}}{2} \sqrt{5}$.  From the above two paragraphs, we have  $a_{-1} > 0$ and $b_{-1} < 0$.   If we set $\frac{w'}{2} + \frac{z'}{2} \sqrt{5} = (\frac{x}{2} + \frac{y}{2} \sqrt{5}) \varepsilon^{-1}$ and if $x > 0$ and $y < 0$, then $w' > 0$ and $y' < 0$.  Thus, $a_n > 0$ and $b_n < 0$ for all $n \leq -1$.

The numbers in $\mathcal{O}$ of norm $p$ are exactly $\pm \frac{a_n}{2} \pm \frac{b_n}{2} \sqrt{5}$ for $n$ in $\mathbb{Z}$.  It follows that the only possible element $\frac{x}{2}+\frac{y}{2} \sqrt{5}$ with norm $p$ and $x > 5y > 0$ other than $\alpha$ is $\frac{a_{-1}}{2} - \frac{b_{-1}}{2} \sqrt{5} = \frac{3a_0 - 5b_0}{4} + \frac{a_0 - 3 b_0}{4} \sqrt{5}$.  But $3 a_0 - 5 b_0 > 5 ( a_0 - 3 b_0)$ implies that $a_0 < 5 b_0$, which we know is not true.  The uniqueness is proved.
\end{proof}

\section{Euclidean algorithm background}\label{S:background}

For positive integers $u > v$, the sequence of equations  of the Euclidean algorithm when commenced by dividing $v$ into $u$ has the form
\begin{align}\label{E:EucAlg}
u &= q_1v+r_1 \notag \\ \notag
v &= q_2r_1+r_2\\ \notag
r_1 &= q_3r_2+r_3\\
&\mathrel{\makebox[\widthof{=}]{\vdots}} \\ \notag
r_{s-3} &=q_{s-1}r_{s-2}+r_{s-1}\\ \notag
r_{s-2} &=q_s r_{s-1} + r_s \notag
\end{align}
with $r_{s-1} = \gcd(u,v)$ and $r_s = 0$.  We define 
\begin{equation*}
	r_{-1}=u \quad \text{ and } \quad r_0=v.
\end{equation*}
Because $r_{s-1} < r_{s-2}$, it follows that $q_s \geq 2$.

Our study of the Euclidean algorithm is streamlined by allowing it to unfold in two different ways.  These parallel the two continued fraction expansions of a rational number.  The expansion of $u/v$ with final quotient $\geq 2$ is the sequence of quotients of the Euclidean algorithm with $u$ and $v$. We will modify the Euclidean algorithm to make it produce the other expansion. If the Euclidean algorithm with $u$ and $v$ is written as \eqref{E:EucAlg}, we replace the final equation by the two equations
\begin{equation}\label{E:EAmodified}
\begin{aligned}
	r_{s-2} &= (q_{s-1} - 1) r_{s-1} + r_{s-1}\\
	r_{s-1} &= 1 \cdot r_{s-1} + 0
\end{aligned}
\end{equation}
This modification changes the parities of the sequences of quotients  and remainders.

\begin{definition}
If $u$ and $v$ are positive integers and $\delta=0$ or $1$, we denote by $\EA(u,v,\delta)$ the sequence  of equations of the Euclidean algorithm when commenced with $u$ and $v$.  When $\delta = 0$, we use whichever of the standard or modified Euclidean algorithms takes an even number of steps,  and when $\delta = 1$, whichever takes an odd number. When considering only the standard algorithm, we write simply $\EA(u,v)$. We denote the $i$th equation by $\EA^i (u,v,\delta)$ or $\EA^i (u,v)$ and call the associated sequences $(q_i)$ and $(r_i)$ the \textbf{sequence of quotients} and \textbf{sequence of remainders}.
\end{definition}

Reasoning about the Euclidean algorithm is facilitated by continuants.  Properties of continuants can be found in Section 6.7 of the book by Graham, Knuth, and Patashnik \cite{GKP1989}.

\begin{definition}
Associated with a sequence $\left[ q_1, \ldots, q_s \right]$ of integers, we define a doubly indexed sequence of \textbf{continuants}
\begin{equation}\label{E:continuantdef}
	\mathfrak{q}_{i,j} = q_i \mathfrak{q}_{i+1, j} + \mathfrak{q}_{i+2,j} \quad \text{ and } \quad \mathfrak{q}_{i+1,i} = 1, \quad \mathfrak{q}_{i+2,i}=0
\end{equation}
for $1 \leq i \leq j+2 \leq s+2$.  When a more explicit description of the $\mathfrak{q}_i$'s is required, we will use the alternate notation (for $i \leq j$):
\begin{equation*}
	\left[ q_i, \ldots, q_j \right] := \mathfrak{q}_{i,j}
\end{equation*}
\end{definition}

The properties of continuants that we will need are the recursion \eqref{E:continuantdef} and the surprising 
\begin{symmetry}
\begin{equation*}
	\left[ q_i, \ldots, q_j \right] = \left[ q_j, \ldots, q_i \right],
\end{equation*}
\end{symmetry}
\noindent which can be proved by induction.  An illuminating combinatorical proof is in \cite{BQ2000}.  From the symmetry of continuants and recurrence \eqref{E:continuantdef} we obtain the alternate recurrence
\begin{equation}\label{E:altrecurrence}
	\mathfrak{q}_{i,j} = q_j \mathfrak{q}_{i,j-1} + \mathfrak{q}_{i,j-2}.
\end{equation}

\begin{lemma}\label{L:endofEA}
Let $u$ and $v$ be relatively prime integers. If $(q_i)_{i=1}^s$ and $(r_i)_{i=-1}^s$ are the sequences of quotients and remainders of $\EA(u,v,\delta)$ and $\mathfrak{q}_{i,j}$ are the continuants corresponding to the sequence of quotients, then
\begin{equation*}
	r_{i} = \mathfrak{q}_{i+2,s}
\end{equation*}
for $i=-1$, \ldots, $s$.  In particular, $u = \mathfrak{q}_{1,s}$ and $v = \mathfrak{q}_{2,s}$.
\end{lemma}

\begin{proof}
Because $u$ and $v$ are relatively prime, we have $r_{s-1} = 1 = \mathfrak{q}_{s+1,s}$ and $r_s = 0 = \mathfrak{q}_{s+2,s}$.  The formula $r_i = \mathfrak{q}_{i+2,s}$ follows from the observation that the recurrence \eqref{E:continuantdef} with $j=s$ is the same recurrence satisfied by the remainders.  
\end{proof}

The continuants $\mathfrak{q}_{1,i}$ have a prominant role in studying the Euclidean algorithm.  They are the numerators of the convergents of the simple continued fraction expansion of $u/v$, and they are the absolute values of coefficients commonly computed as part of the extended Euclidean algorithm.  We therefore make the following definition.

\begin{definition}
Let $q_1$, $q_2$, \ldots, $q_s$ be the sequence of quotients of $\EA(u,v, \delta)$ with associated continuants $\mathfrak{q}_{i,j}$.  We define the \textbf{Bezout coefficients} of $u$ and $v$ by 
\begin{equation*}
	\beta_i = \mathfrak{q}_{1,i}
\end{equation*}
for $-1 \leq i \leq s$.  
\end{definition}

The following lemmas reveal a close connection between the sequence of remainders of $\EA(u,v, \delta)$ and the corresponding Bezout coefficients.  Each makes a fine exercise in mathematical induction.

\begin{lemma}\label{L:congruence}
If $(q_i)_{i=1}^s$ and $(r_i)_{i=-1}^s$ are the sequences of quotients and remainders of $\EA(u,v, \delta)$ and $(\beta_i)_{i=-1}^{s}$ are the Bezout coefficients, then 
\begin{equation*}
	v \beta_i \equiv (-1)^i r_i \pmod{u} \text{ for $-1 \leq i \leq s$}
\end{equation*}
\end{lemma}

\begin{proof}
The cases $i=-1$ and $i=0$ simply say that $0 \equiv -u \pmod{u}$  and $v \equiv v \pmod{u}$.  Further, if the congruence holds for $i-1$ and $i$ with $0 \leq i \leq s-1$, then
\begin{align*}
	v \beta_{i+1} &= v q_{i+1} \beta_i  + v \beta_{i-1} \\
	&\equiv (-1)^{i} q_{i+1} r_{i} + (-1)^{i-1} r_{i-1}  \pmod{u}\\
	&= (-1)^{i+1} r_{i+1}.
\end{align*}
The lemma follows by induction.
\end{proof}

\begin{lemma}\label{L:linearcombination}
If $(q_i)_{i=1}^{s}$ and $(r_i)_{i=-1}^{s}$ are the sequences of quotients and remainders of $\EA(u,v, \delta)$ and $\left( \beta_i \right)_{i=-1}^{s}$ are the Bezout coefficients, then $u=\beta_i r_{i-1}+\beta_{i-1}r_i$ for $0 \leq i \leq s$.
\end{lemma}

\begin{proof}
For $i=0$, the equation is just $u=u$.
Assume that $u=\beta_ir_{i-1}+\beta_{i-1}r_{i}$ for some $i$ with $0 \leq i \leq s-1$. Then using \eqref{E:altrecurrence},
\begin{equation*}
	u=\beta_i(q_{i+1} r_i+r_{i+1})+(\beta_{i+1}-q_{i+1} \beta_i)r_i=\beta_{i+1}r_i+\beta_ir_{i+1}.
\end{equation*}
The lemma follows by induction.
\end{proof}

We now discuss background for studying structure in the Euclidean algorithm quotients.  
Fix a positive integer $k$. In recent work \cite{bS2015}, the third author proved that if $v$ with $0 < v < u$ satisfies the congruence
\begin{equation*}
	v^2 + kv \pm 1 \equiv 0 \pmod{u},
\end{equation*}
then the sequence of quotients of $\EA(u,v,\delta)$ (with $\delta=0$ if the plus sign is used in the above congruence and $\delta=-1$ otherwise) fits one of a finite list of ``end-symmetric'' patterns.  The list of patterns depends only on $k$.  We will use this result only when $k=1$, $2$, or $3$.

\begin{lemma}\label{L:patterns}
The sequence of quotients of $\EA(u,v,1)$ when $v^2 +  v - 1 \equiv 0 \pmod{u}$ has the form
\begin{equation*}
	\begin{matrix}
		q_1, & \ldots & q_{s-1}, & q_s + (-1)^{s+1}, & 1, & q_s, & q_{s-1}, & \ldots & q_1
	\end{matrix}
\end{equation*} 
for some positive integers $q_1$, \ldots, $q_s$.

When $v^2 + 3 v + 1 \equiv 0 \pmod{u}$, then $\EA(u,v,0)$ has quotient sequence of the form
\begin{equation*}
	\begin{matrix}
		q_1, & \ldots & q_{s-1}, & q_s + (-1)^{s+1} \cdot 3, & q_s, & q_{s-1}, & \ldots & q_1
	\end{matrix}
\end{equation*} 
for some positive integers $q_1$, \ldots, $q_s$.

When $v^2 + 2v + 1 \equiv 0 \pmod{u}$, that is, when
\begin{equation}\label{E:squarecongruence}
	(v+(-1)^{\delta})^2 \equiv 0 \pmod{u},
\end{equation}
then $\EA(u,v,0)$ has quotient sequence fitting one of the patterns
\begin{equation}\label{E:evenpatterns}
	\begin{matrix}
		& q_1, & \ldots & q_{s-1}, & {q_s +(-1)^{s+1} \cdot 2},  & {q_s}, & q_{s-1}, & \ldots & q_1 \\
		q_1, & \ldots & q_{s-1}, & {q_s + 1}, & {x}, & {1},   & {q_s}, & q_{s-1}, & \ldots & q_1\\
		q_1, & \ldots & q_{s-1}, & {q_s  - 1}, & {1}, & {x},   & {q_s}, & q_{s-1}, & \ldots & q_1\\
	\end{matrix}.
\end{equation}
for some positive integers $q_1$, \ldots, $q_s$ and $x$.  
\end{lemma}

The patterns \eqref{E:evenpatterns} are well known, being related to paper-folding sequences and folded continued fractions \cite{jS1979,aV2001}.  What seems to be new is their appearance the quotients of the Euclidean algorithm with $u$ and $v$ when $v$ satisfies \eqref{E:squarecongruence}.  Theorem \ref{T:degenerate} gives an arithmetical criteria for deciding which of the patterns \eqref{E:evenpatterns} describes the simple continued fraction expansion of $u/v$.

\vspace{0.5cm}

\section{Explicating the Euclidean algorithm}\label{S:classification}

Suppose $u$ and $v$ are positive integers with $u > v$ and $v^2 + v - 1 \equiv 0 \pmod{u}$.  Then $v-1$ satisfies the congruence $v^2 + 3v + 1 \equiv 0 \pmod{u}$.  According to Lemma \ref{L:patterns}, $\EA(u,v,1)$ has sequence of quotients of the form $q_1$, \ldots, $q_s + \delta_1$, $1$, $q_s+\delta_0$, \ldots, $q_1$, while $\EA(u,v-1,0)$ has sequence of quotients of the form $\tilde{q}_1$, \ldots, $\tilde{q}_s + \delta_1 \cdot 3$, $\tilde{q}_s + \delta_0 \cdot 3$, \ldots, $\tilde{q}_1$.  In both cases, $\delta_1 = 1$ if $s$ is odd and $0$ if $s$ is even, while $\delta_0 = 1$ if $s$ is even and $0$ if $s$ is odd. There is no \emph{a priori} reason for the sequence of $q_i$'s to equal the sequence of $\tilde{q}_i$'s.  Nevertheless, that is the conclusion of the following theorem, which also gives explicit formulas for the remainders of $\EA(u,v-1,0)$ in terms of the remainders of $\EA(u,v,1)$.

\begin{theorem}\label{T:twoEAs}
Let $u$ and $v$ be positive integers $u > v$, with $v^2 + v - 1 \equiv 0 \pmod{u}$.  Write the sequence of quotients of $\EA(u,v,1)$ as 
\begin{equation*}
	q_1, \ldots, q_s + \delta_1, 1, q_s + \delta_0, \ldots, q_1.
\end{equation*}
Let $(r_i)_{i=-1}^{2s+1}$ be the sequence of remainders, and for $i=-1$, \ldots, $s-1$, set $t_i = r_i + (-1)^{i+1} r_{2s-i}$.  Then $EA(u,v-1,0)$ is the sequence of $2s$ equations
\begin{alignat*}{3}
t_{i-2} \hspace{0.2cm} &= \hspace{1.1cm} q_{i} \cdot t_{i-1} \quad & &+ \hspace{0.45cm} t_i \quad & &\text{for $1 \leq i \leq s-1$}\\
t_{s-2} \hspace{0.2cm} &= \hspace{0.3cm} (q_s + \delta_1 \cdot 3) \cdot t_{s-1}  \quad &&+  \hspace{0.3cm} r_{s+1} & &\\
t_{s-1} \hspace{0.2cm} &= \hspace{0.3cm} (q_s + \delta_0 \cdot 3) \cdot r_{s+1} \quad &&+  \hspace{0.3cm} r_{s+2} & &\\
r_{i-1} \hspace{0.2cm} &= \hspace{0.7cm} q_{2s+1-i} \, \cdot r_{i} \hspace{-0.3cm} &&+  \hspace{0.3cm} r_{i+1} \quad & &\text{for $s+2 \leq i \leq 2s$}
\end{alignat*}
\end{theorem}

\begin{proof}
A quick check verifies that $t_{-1} = u$ and $t_0 = v-1$, which begin the remainder sequence of $\EA(u,v-1,0)$. Because the sequence $(r_i)_{i=1}^{2s+1}$ is decreasing, it is clear that the purported quotients and remainders are all positive.  We check that the purported remainders form a strictly decreasing sequence (except that the final two may be equal when $\EA(u,v-1,0)$ is computed using the modification \eqref{E:EAmodified} of the Euclidean algorithm.) This is apparent for $r_{s+1}$, \ldots, $r_{2s+1}$. Also,  $t_{s-1} \geq r_{s-1} - r_{s+1} = r_s > r_{s+1}$.  (The equality is because the middle quotient of $\EA(u,v,1)$ is $1$.)

We must show $t_{i} > t_{i+1}$ for $1 \leq i \leq s-2$.  From the division algorithm, we have $r_i \geq r_{i+1} + r_{i+2}$ for $-1 \leq i \leq 2s-1$.  Thus, for $-1 \leq i \leq s-3$, we have
\begin{align*}
	r_{i} - r_{i+1} &\geq r_{i+2} \geq r_{i+3} + r_{i+4} >  r_{2s-i} + r_{2s-i-1}.
\end{align*}
It follows that $t_{i} > t_{i+1}$ for $1 \leq i \leq s-3$.  The above chain of inequalities also holds with the final inequality replaced by an equality when $i=s-2$.  The second inequality is strict when $i=s-2$ unless  $q_s + \delta_0 = 1$, which only happens if $s$ is odd.  But in that case, $t_{s-2} = r_{s-2} + r_{s+2} > r_{s-1} - r_{s+1} = t_{s-1}$ holds anyway.  

To ensure the equations in the theorem are the steps of $\EA(u,v-1,0)$, it remains to check the algebraic validity of each step.  The theorem will then follow from the uniqueness of the quotients and remainders.

The equation $t_{i-2} = q_i \cdot t_{i-1} + t_i$ is equivalent to
\begin{equation*}
	(-1)^{i+1} \left( r_{i-2} - q_i r_{i-1} - r_i \right) =     r_{2s-i}  - q_i  r_{2s+1-i}  - r_{2s+2-i}
\end{equation*}
The expression on the left is $0$.  Also, examining the pattern of the sequence of quotients of $\EA(u,v,1)$, we see that $q_{2s+2-i} = q_i$ for $i=1$, \ldots, $s-1$.  Thus, the $2s-i+1$th step of $\EA(u,v,1)$ is 
\begin{equation}\label{E:algebracheck}
	r_{2s-i} = q_i r_{2s+1-i} + r_{2s+2-i},
\end{equation}
and the right side is also 0. Substituting $2s+1-i$ for $i$ in \eqref{E:algebracheck}, we find as well that $r_{i-1} = q_{2s+1-i} r_i + r_{i+1}$ for $s+2 \leq i \leq 2s$, which verifies steps $i=s+2$ through $i=2s$ in the theorem.

We now check the middle pair of equations.  We know that the $s$th through $s+2$nd equations of $\EA(u,v,1)$ are
\begin{alignat}{2}\label{E:middle}
	r_{s-2} &= \quad\left( q_s + \delta_1 \right) r_{s-1} & &+ \quad r_s \notag\\
	r_{s-1} &= \quad \quad \qquad r_s & &+ \quad r_{s+1} \\
	r_{s} &= \quad \left( q_s + \delta_0 \right) r_{s+1} & &+ \quad r_{s+2}. \notag
\end{alignat}
Assume first that $s$ is odd so that $\delta_1 = 1$ and $\delta_0 = 0$. The equation $t_{s-2} = \left( q_s + \delta_1 \cdot 3 \right) t_{s-1} + r_{s+1}$ is equivalent to
\begin{equation*}
	r_{s-2}  = \left( q_s + 3 \right) \left( r_{s-1}  - r_{s+1} \right) + r_{s+1} - r_{s+2}.
\end{equation*}
Substituting in turn $r_{s+2} = r_s - q_s r_{s+1}$ and $r_{s+1} = r_{s-1} - r_{s}$ from \eqref{E:middle}, this is equivalent to
\begin{align*}
	r_{s-2}  &= \left( q_s + 3 \right) \left( r_{s-1} - r_{s+1} \right) + r_{s+1} - r_s + q_s r_{s+1} \\
	&= \left( q_s + 3 \right) r_s + r_{s-1} - 2r_s + q_s r_{s-1} - q_s r_s\\
	&= (q_s + 1) r_{s-1} + r_s,
\end{align*}
which is the first of equations \eqref{E:middle}.  

If, instead, $s$ is even, so $\delta_1 = 0$ and $\delta_0 = 1$, then $t_{s-2} = (q_s + \delta_1 \cdot 3) t_{s-1} + r_{s+1}$ is equivalent to
\begin{equation*}
	r_{s-2} = q_s \left( r_{s-1} + r_{s+1} \right) + r_{s+1} + r_{s+2}
\end{equation*}
Substituting in turn $r_{s+2} = r_s - q_s r_{s+1} - r_{s+1}$ and $r_{s+1} = r_{s-1} - r_s$, this is equivalent to
\begin{align*}
	r_{s-2} &= q_s \left( r_{s-1} + r_{s+1} \right) + r_s - q_s r_{s+1}\\
	&= q_s \left( 2 r_{s-1} - r_s \right) + r_s - q_s r_{s-1} + q_s r_s\\
	&= q_s r_{s-1} + r_s,
\end{align*}
which is the first of equations \eqref{E:middle}.

The verification that $t_{s-1} = (q_s + \delta_0 \cdot 3) \cdot r_{s+1} + r_{s+2}$ is entirely similar, using the latter two equations of \eqref{E:middle}.
\end{proof}

\begin{proof}[Proof of Algorithm \ref{A:BQFalgorithm}]
Let the quotients and remainders of $\EA(u,v,1)$ be written as in Theorem \ref{T:twoEAs}. Suppose first that $s$ is odd.  Applying Lemma \ref{L:linearcombination} with $i=s$ to $\EA(u,v,1)$, we have $u = \left[ q_1, \ldots, q_{s-1}, q_s + 1 \right]  r_{s-1} + \left[ q_1, \ldots, q_{s-1} \right] r_s$.    By the symmetry of continuants and recurrence \eqref{E:altrecurrence}, it follows that
\begin{align*}
	u &= \left[ q_s + 1, q_{s-1}, \ldots, q_1 \right] r_{s-1}  + \left[ q_{s-1}, \ldots, q_1 \right] r_s\\
	&=   \left[ q_{s-1}, \ldots, q_1 \right] \left( r_{s-1} + r_s \right) + \left[ q_{s}, \ldots, q_1 \right]  r_{s-1}
\end{align*}
Now use the ``end-symmetric'' form of the quotient sequence of $\EA(u,v,1)$ and Lemma \ref{L:endofEA} to obtain
\begin{equation*}
	u = r_{s+1} \left( r_{s-1} + r_s \right) + r_{s} r_{s-1}
\end{equation*}
Substituting out $r_{s-1}$ using the middle of equations \eqref{E:middle} gives
\begin{equation*}
	u = r_s^2 + 3 r_s r_{s+1} + r_{s+1}^2
\end{equation*}

Suppose now that $s$ is even.  Applying Lemma \ref{L:linearcombination} with $i=s$ to $\EA(u,v,1)$ in this case gives $u = \left[ q_1, \ldots, q_s \right] r_{s-1} + \left[ q_1, \ldots, q_{s-1} \right] r_s$.  Again using the recurrence \eqref{E:altrecurrence}, it follows that
\begin{equation*}
	u =  \left[ q_s + 1, q_{s-1}, \ldots, q_1 \right] r_{s-1} +  \left[ q_{s-1}, \ldots, q_1 \right] \left( r_s - r_{s-1} \right),
\end{equation*}
and Lemma \ref{L:endofEA} shows
\begin{equation*}
	u = r_s r_{s-1} + r_{s+1} \left( r_s - r_{s-1} \right).
\end{equation*}
Substituting with \eqref{E:middle} once more gives
\begin{equation*}
	u = (r_s - r_{s+1})^2 + 3 (r_s - r_{s+1}) r_{s+1} + r_{s+1}^2
\end{equation*}
Thus, in either case, $r_{s+1} = c$ in the unique representation $p = b^2 + 3bc + c^2$ with $b> c > 0$. If $s$ is odd, then $r_s = b$, and if $s$ is even, then $r_s = b+c$.  The inequalities $5 b^2 > b^2 + 3bc + c^2 > 5 c^2$ show that 
\begin{equation*}
	b + c > b > \sqrt{\frac{p}{5}} > c
\end{equation*}
Thus, regardless of whether $s$ is odd or even, $c$ is the first remainder smaller than $\sqrt{\frac{p}{5}}$.
\end{proof}

Fix anew positive integers $b$ and $c$ with $\gcd(b,c)=1$.  We next give an explicit description of the quotients and remainders of $\EA(b^2, bc \pm 1)$ in terms of the quotients, remainders, and Bezout coefficients of $\EA(b,c)$.   The algorithm for computing inverses in modular arithmetic falls out of this description. 

\begin{theorem}\label{T:theorem1}
Let $b > c > 1$ be integers with $\gcd(b,c)=1$.  Let $(q_i)_{i=1}^s$ and $(r_i)_{i=-1}^s$ be the sequences of quotients and remainders of the standard (i.e., unmodified) Euclidean algorithm with $b$ and $c$, let $(\beta_i)_{i=-1}^{s}$ be the corresponding continuants, and set  $t_i=r_i b  \pm(-1)^i\beta_i$ for $-1 \leq i \leq s-1$. Then $EA(b^2,bc\pm1, 0)$ is the sequence of $2s$ equations
\begin{align*}
t_{i-2} \hspace{0.2cm} &= \hspace{1.0cm} q_{i} \cdot t_{i-1} &&+ \hspace{0.1cm} t_{i} \quad & &\text{for $1 \leq i \leq s-1$}\\
t_{s-2} \hspace{0.2cm} &= \hspace{0.35cm} (q_s \pm(-1)^{s}) \cdot t_{s-1} \hspace{-0.4cm} &&+  \hspace{0.1cm}\beta_{s-1} & &\\
t_{s-1} \hspace{0.2cm} &= \hspace{0.2cm} (q_s \pm(-1)^{s-1}) \cdot \beta_{s-1} \hspace{-0.8cm} &&+  \hspace{0.1cm} \beta_{s-2} & &\\
\beta_{2s+1-i} \hspace{0.2cm} &= \hspace{0.6cm} q_{2s+1-i} \, \cdot \beta_{2s-i} \hspace{-0.3cm} &&+  \hspace{0.1cm} \beta_{2s-1-i} \quad & &\text{for $s+2 \leq i \leq 2s$}
\end{align*}
\end{theorem}

\begin{proof} 
The proof can be conducted in an analogous manner to the proof of Theorem \ref{T:twoEAs}.
One readily checks that the first two remainders are $t_{-1} = b^2$ and $t_0 = bc \pm 1$. The observation $q_s \geq 2$ was made in the first paragraph of Section \ref{S:background}, so the purported quotients are all positive.  So are the remainders since $b \geq \beta_i$ for $-1 \leq i \leq s-1$.

For $s+2 \leq i \leq 2s$, the equation $\beta_{2s+1-i} = q_{2s+1-i} \, \cdot \beta_{2s-i} + \beta_{2s-1-i}$ follows from \eqref{E:altrecurrence}.  For $1 \leq i \leq s-1$, the equality $t_{i-2} = q_i t_{i-1} + t_i$ can be deduced from the equation $\EA^i(b,c)$ and \eqref{E:altrecurrence}.  To verify the middle two equations, we  first note that because $b$ and $c$ are relatively prime, we have $r_{s-1} = 1$, $t_{s-1} = b \pm (-1)^{s-1} \beta_{s-1}$, and $q_s = r_{s-2}$.   The equations can then be verified using Lemma \ref{L:linearcombination} with $u=b$, $v=c$, and $i=s-1$:
\begin{align*}
	(q_s \pm (-1)^{s}) t_{s-1} + \beta_{s-1} &= (r_{s-2} \pm (-1)^{s}) b \pm  (-1)^{s-1} r_{s-2} \beta_{s-1} \\
	&= r_{s-2}  \, b \pm (-1)^{s-2} \beta_{s-2} \\ &= t_{s-2}
\end{align*}
and
\begin{align*}
	(q_s \pm(-1)^{s-1})\beta_{s-1} +\beta_{s-2} &= r_{s-2} \beta_{s-1} \pm (-1)^{s-1} \beta_{s-1} + \beta_{s-2}\\
	&= b \pm (-1)^{s-1} \beta_{s-1} \\
	&= t_{s-1}.
\end{align*}

Finally, the remainders form a decreasing sequence.  For $-1 < i < s-1$, the inequality $\left( r_{i} - r_{i+1} \right) n > \beta_{i} +\beta_{i+1}$ follows from Lemma \ref{L:linearcombination} and implies that $t_{i} > t_{i+1}$.  The inequality $\beta_{s-1} < t_{s-1}$ follows from the equation $t_{s-1} = \left( q_s \pm (-1)^{s-1} \right) \beta_{s-1} + \beta_{s-2}$ verified in the last paragraph.  And $\beta_{i-1} < \beta_{i}$ for $0 \leq i \leq s$ follows from the recurrence \eqref{E:altrecurrence}.
\end{proof}

\begin{proof}[Proof of the algorithm for multiplicative inverses]
When $m=1$, the algorithm is easily validated.  If $m > n$, then the third step of $\EA (n^2, mn+ 1)$ will be division of $rn + 1$ into $n^2$, where $r$ is the remainder when $m$ is divided by $n$.  Thus, it suffices to assume $n > m > 1$, so also $s > 1$.

Theorem \ref{T:theorem1} implies the first remainder less than $n$ in $\EA(n^2, mn+1)$ is $\beta_{s-1}$ when $s$ is odd and $t_{s-1}$ when $s$ is even.  We apply Lemma \ref{L:congruence} to $\EA (n,m)$ to find $m \beta_{s-1} \equiv (-1)^{s-1} \pmod{n}$.  Thus when $s$ is odd, the product of $m$ and the first remainder less than $n$ is
\begin{equation*}
	m \beta_{s-1} \equiv 1 \pmod{n}.
\end{equation*}
When $s$ is even, the product is
\begin{equation*}
	m t_{s-1} = mn - m \beta_{s-1} \equiv 1 \pmod{n}. \qedhere
\end{equation*}
\end{proof}

We now give a complete description of $\EA(ab^2, abc \pm 1)$ for positive integers $a \geq 2$, $b$, and $c$ and $\gcd (b,c) = 1$.

\begin{theorem}\label{T:theorem4}
Let $a$, $b$, $c$, and $k$ be integers with $b > c > 1$, $\gcd(b,c)=1$, and $a \geq 2$.  Let $(q_i)_{i=1}^s$ and $(r_i)_{i=-1}^s$ be the sequences of quotients and remainders in $\EA(b,c)$, let $(\beta_i)_{i=-1}^{s}$ be the corresponding Bezout coefficients, and set $t_i=a b r_i  + (-1)^{i+k} \beta_i$ for $-1 \leq i \leq s-1$. If $(-1)^{s+k}=-1$, then $EA(a b^2, abc + (-1)^k,0)$ is the sequence of $2s+2$ equations
\begin{align*}
t_{i-2} \hspace{0.2cm} &= \hspace{0.8cm} q_{i} \cdot t_{i-1} \hspace{-0.4cm}  &&+  \hspace{0.1cm} t_{i} \quad & &\text{for $1 \leq i \leq s-1$}\\
t_{s-2} \hspace{0.2cm}  &= \hspace{0.35cm}  (q_s-1) \cdot t_{s-1} \hspace{-0.4cm}  &&+  \hspace{0.1cm} (t_{s-1}-b) &&\\
t_{s-1} \hspace{0.2cm}  &= \hspace{0.25cm} 1 \cdot (t_{s-1}- b) \hspace{-0.4cm}  &&+  \hspace{0.1cm} b &&\\
t_{s-1}-b \hspace{0.2cm} &= \hspace{0.5cm} (a - 1) \cdot b \hspace{-0.4cm}  &&+  \hspace{0.1cm} \beta_{s-1} && \\ 
b \hspace{0.2cm}  &= \hspace{0.85cm} q_s \cdot \beta_{s-1}  \hspace{-0.4cm}  &&+  \hspace{0.1cm} \beta_{s-2} && \\ 
\beta_{2s+3-i} \hspace{0.2cm} &= \hspace{0.2cm} q_{2s+3-i} \cdot \beta_{2s+2-i} \hspace{-0.3cm}  \hspace{-0.4cm}  &&+  \hspace{0.1cm} \beta_{2s+1-i} \quad & &\text{for $s+4 \leq i \leq 2s+2$.}
\end{align*}
When $(-1)^{s+k}=1$, steps $s$ through $s+3$ change to:
\begin{align*}
t_{s-2}  \hspace{0.2cm} &= \hspace{0.70cm} q_s \cdot  t_{s-1} \hspace{0.05cm} &&+ \hspace{0.1cm} b &&\\
t_{s-1}   \hspace{0.2cm} &= \hspace{0.5cm} (a-1) \cdot b \hspace{0.05cm} &&+ \hspace{0.1cm} (b-\beta_{s-1}) &&\\
b  \hspace{0.2cm} &= \hspace{0.25cm} 1\cdot (b - \beta_{s-1}) \hspace{0.05cm} &&+ \hspace{0.1cm} \beta_{s-1} &&\\
b-\beta_{s-1}   \hspace{0.2cm} &= \hspace{0.2cm} (q_s-1) \cdot \beta_{s-1} \hspace{0.05cm} &&+ \hspace{0.1cm} \beta_{s-2} \quad & & \hspace{0.35cm} \phantom{\text{for $s+4 \leq i \leq 2s+2$.}}\\
\end{align*}
\end{theorem}

\begin{proof}

It follows as in the proof of Theorem \ref{T:theorem1} that the purported quotients and remainders are positive (excluding the final remainder). The equations $\beta_{2s+3-i} = q_{2s+3-i} \, \cdot \beta_{2s+2-i} +\beta_{2s+1-i}$ and $t_{i-2} = q_i t_{i-1} + t_i$ can be deduced as in the proof of Theorem \ref{T:theorem1}.  The equations $t_{s-1} = 1 \cdot \left( t_{s-1} - b \right) + b$ and $b=1 \cdot (b - \beta_{s-1}) + \beta_{s-1}$ are clearly true.     Lemma \ref{L:endofEA} shows that $\beta_s = b$.  Thus, the equations $b = q_s \cdot \beta_{s-1} + \beta_{s-2}$ and $b - \beta_{s-1} = (q_s - 1) \beta_{s-1} + \beta_{s-2}$ are consequences of \eqref{E:altrecurrence}. 

Since $\gcd (b,c) = 1$, we have $r_{s-1} = 1$, $t_{s-1} = a b - (-1)^{s+k} \beta_{s-1}$, and $q_s = r_{s-2}$.  From this, we obtain the equations $t_{s-1} - b = (a-1) b + \beta_{s-1}$ when $(-1)^{s+k}=-1$ and $t_{s-1} = (a-1) b + (b - \beta_{s-1})$ when $(-1)^{s+k}=1$. 

When $(-1)^{s+k} = -1$, the $s$th equation is valid since
\begin{align*}
	(q_s-1) t_{s-1} + (t_{s-1} - b) &= q_s (a b + \beta_{s-1}) - \beta_s\\
	&= a b r_{s-2} +  (\beta_s - \beta_{s-2}) - \beta_s,\\
	&= t_{s-2}.
\end{align*}
Similarly, when $(-1)^{s+k}=1$,
\begin{align*}
	q_s t_{s-1} + b &= q_s (a b -\beta_{s-1}) + b\\
	&= a b r_{s-2}  - (\beta_s - \beta_{s-2}) + \beta_s\\
	&= t_{s-2}.
\end{align*}  

When $(-1)^{s+k} = -1$, the inequality $t_{s-1} - b < t_{s-1}$ is clear and the inequality $b < t_{s-1} - b$ follows from the assumption that $a \geq 2$.  When $(-1)^{s+k} = 1$, the inequality $b < t_{s-1}$ follows from the assumption that $a \geq 2$ and from $b = \beta_s > \beta_{s-1}$.  The inequality $b-\beta_{s-1} < b$ is clear, and the inequality $\beta_{s-1} < b - \beta_{s-1}$ follows from $b = q_s \beta_{s-1} + \beta_{s-2}$ and $q_s \geq 2$. 
That $t_i < t_{i-1}$ and $\beta_i > \beta_{i-1}$ for $1 \leq i \leq s-1$ follows as in the proof of Theorem \ref{T:theorem1}.
\end{proof}

To conclude, we provide an arithmetical characterization of which quotient pattern will appear when performing the Euclidean algorithm with $u$ and $v$ with $(v \pm 1)^2 \equiv 0 \pmod{u}$.

\begin{theorem}\label{T:degenerate}
Let $u$ be a positive integer and write $u=a b^2 $, where $a$ is the square free part of $u$. Assume $v$ with $0 < v < u$ satisfies $(v+(-1)^{\delta})^2 \equiv 0 \pmod{u}$.  Then there is an integer $c$ such that
\begin{equation*}
		v = abc + (-1)^{\delta+1}
\end{equation*}
The continued fraction expansion of $u/v$ with even length has quotient sequence fitting the first of the patterns \eqref{E:evenpatterns}  if and only if $\gcd(b,c) = a = 1$. Otherwise, it fits one of the other patterns with $x = \gcd(b,c)^2 \cdot a - 1$.  The second pattern appears if $s+\delta$ is even, and the third if $s+\delta$ is odd.  In all cases, $q_0$, \ldots, $q_s$ is the quotient sequence of the continued fraction expansion of $b/c$
\end{theorem}

\begin{proof}
By assumption, there exists some integer $w$ such that $(v + (-1)^{\delta})^2 = uw$.  Consideration of prime factorizations shows that $a$ is also the square free part of $w$, say $w = a c^2$.  Then $v = abc + (-1)^{\delta+1}$.

 If $\gcd(b,c)=d$ and we set $\tilde{a} = a d^2$, $\tilde{b} = \tfrac{b}{d}$, and $\tilde{c} = \tfrac{c}{d}$, then 
\begin{equation*}
	u =  \tilde{a} \tilde{b}^2, \quad v = \tilde{a} \tilde{b} \tilde{c} + (-1)^{\delta + 1}, \quad \text{and $\gcd (\tilde{b}, \tilde{c}) = 1$.}
\end{equation*}
Theorem \ref{T:degenerate} now follows from Theorem \ref{T:theorem1} and Theorem \ref{T:theorem4}.
\end{proof}

\end{document}